\documentclass[12pt,a4]{article}
\usepackage{amsmath, amsthm, amssymb}
\usepackage[all]{xy}
\usepackage{amscd}
\usepackage{array}
\usepackage{multirow}
\usepackage{color}
\usepackage[affil-it]{authblk}
\usepackage{rotating}
\usepackage{mathrsfs}
\usepackage{textcomp}
\usepackage{graphicx}
\usepackage[a4paper,left=2.5cm,right=2.5cm,top=2.5cm,bottom=2.5cm]{geometry}
\newcommand{\tdot}{\includegraphics[width=0.125in]{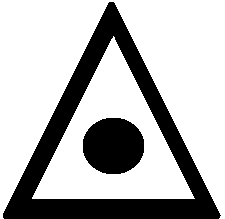}}

\numberwithin{equation}{section}
\newtheorem{thm}{Theorem}[section]
\newtheorem{defnt}[thm]{Definition}

\newtheorem{prop}[thm]{Proposition}

\newtheorem{cor}[thm]{Corollary}
\newtheorem{lemma}[thm]{Lemma}
\newtheorem{question}[thm]{Question}

\newenvironment{proof*}{\begin{trivlist}\item[\hskip\labelsep{\em Proof of Theorem \ref{t2}}]}{%
\hfill$\square$\rm\end{trivlist}}
\newenvironment{proof**}{\begin{trivlist}\item[\hskip\labelsep{\em Proof of Proposition \ref{p4}}]}{%
\hfill$\square$\rm\end{trivlist}}

\begin{document}
\bibliographystyle{plain}
\pagestyle{plain} \pagenumbering{arabic}
\date{}
\newpage
\title{\textbf{\small{Some Remarks on Diametral Dimension and Approximate Diametral Dimension of Certain Nuclear Fr\'echet Spaces}}}
\author{
\small{Nazl{\i} Do\u{g}an}
}
 \affil{{ \small{{Department of Mathematics} \\
{Istanbul Technical University}\\
{Maslak, 34469, Istanbul, Turkey}\\

}}}

\maketitle

\begin{abstract}\footnotesize{The diametral dimension, $\Delta(E)$, and the approximate diametral dimension, $\delta (E)$,  of a nuclear Fr\'echet space $E$ which satisfies $\underline{DN}$ and $\Omega$, are related to corresponding invariant of power series spaces $\Lambda_{1}(\varepsilon)$ and $\Lambda_{\infty}\left(\varepsilon\right)$ for some exponent sequence $\varepsilon$. In this article, we examine a question of whether $\delta (E)$ must coincide with that of a power series space if $\Delta(E)$ does the same, and vice versa. In this regard, we first show that this question has an affirmative answer in the infinite type case by showing that $\Delta (E)=\Delta\left(\Lambda_{\infty} (\varepsilon)\right)$ if and only if $\delta (E)= \delta (\Lambda_{\infty}(\varepsilon))$. Then we consider the question in the finite type case and, among other things, we prove that  $\delta (E)=\delta\left(\Lambda_{1} (\varepsilon)\right)$ if and only if $\Delta (E)= \Delta (\Lambda_{1}(\varepsilon))$ and $E$ has a prominent bounded subset.}

\end{abstract}

\footnotesize{{\bf \emph{Keywords}:} \emph{Nuclear Fr\'echet Spaces, Diametral Dimension, Topological Invariants, Prominent Bounded Subsets\\2010 Mathematics Subject Classification. 46A04, 46A11, 46A63.}}
%%%%%%%%%%%%%%%%%%%%%%%%%%%%%%%%%%%%%%%%%%%%%%%%%%%%%%%%%%%%%%%%%%
\section{Introduction}\label{intro}
Power series spaces constitute an important and well studied class in the theory of Fr\'echet spaces. Linear topological invariants $\underline{DN}$ and $\Omega$ (see definition below) are enjoyed by many natural nuclear Fr\'echet spaces appearing in analysis. In particular, spaces of analytic functions, solutions of homogeneous elliptic linear partial differential operators with their natural topologies have the properties $\underline{DN}$ and $\Omega$, see \cite{MV}  and \cite{T6}. 
~\\ Let $E$ be a nuclear Fr\'echet space which satisfies $\underline{DN}$ and $\Omega$. Then it is a well known fact that the diametral dimension $\Delta (E)$ and the approximate diametral dimension $\delta (E)$ of $E$ are set theoretically between corresponding invariant of power series spaces $\Lambda_{1}(\varepsilon)$ and $\Lambda_{\infty}\left(\varepsilon\right)$ for some specific exponent sequence $\varepsilon$. Coincidence of diametral dimension and/or approximate diametral dimension of $E$ with that of a power series space yields some structural results. For example, in \cite{AKT2}, Aytuna et al. proved that a nuclear Fr\'echet space $E$ with the properties $\underline{DN}$ and $\Omega$ contains a complemented copy of $\Lambda_{\infty}\left(\varepsilon\right)$ provided the diametral dimensions of $E$ and $\Lambda_{\infty}\left(\varepsilon\right)$ are equal and $\varepsilon$ is stable. On the other hand, Aytuna \cite{A} characterized tame nuclear Fr\'echet spaces $E$ with the properties $\underline{DN}$, $\Omega$ and  stable associated exponent sequence $\varepsilon$, as those that satisfies  $\delta (E)= \delta (\Lambda_{1}(\varepsilon))$. These results leads to ask the following question:

\begin{question}\label{q1} Let $E$ be a nuclear Fr\'echet space with the properties $\underline{DN}$ and $\Omega$. If diametral dimension of $E$ coincides with that of a power series space, then does this imply that the approximate diametral dimension is also do the same and vice versa?
\end{question}

% Motivated by these results,in the class of nuclear Fr\'echet spaces with the properties $\underline{DN}$ and $\Omega$, we investigate diametral dimension and approximate diametral dimension if one of them is equivalent to that of a power series space. We first show that if one of the invariants coincides with that of an infinite type power series space this forces the other invariant to do the same.  On the other hand, if the power series space in question is of finite type, then the existence of a prominent bounded set in the nuclear space plays a decisive role. In particular, we show that the approximate diametral dimension of a nuclear Fr\'echet space $E$ with the properties $\underline{DN}$ and $\Omega$ is equal to the approximate diametral dimension of a finite type power series space $\Lambda_{1}\left(\varepsilon\right)$ if and only if the diametral dimensions of $E$ and $\Lambda_{1}\left(\varepsilon\right)$ coincide and $E$ has a prominent bounded set. 
This article is concerned with this question and the layout is as follows:
~\\ In Section 2, we give some preliminary materials. Then in Section 3, we show that Question \ref{q1} has an affirmative answer when the power series space is of infinite type. In our final section, we search an answer for the Question \ref{q1} in the finite type case and, in this regard, we first prove that the condition $\delta\left(E\right)= \delta\left(\Lambda_{1}\left(\varepsilon\right)\right)$ always implies $\Delta\left(E\right)= \Delta\left(\Lambda_{1}\left(\varepsilon\right)\right)$. For other direction, the existence of a prominent bounded subset in the nuclear Fr\'echet space $E$ plays a decisive role for the answer of Question 1.1. Among other things, we prove that $\delta\left(E\right)=\delta\left(\Lambda_{1}\left(\varepsilon\right)\right)$ if and only if $E$ has a prominent bounded set and $\Delta\left(E\right)=\Delta(\Lambda_{1}\left(\varepsilon\right))$.

%%%%%%%%%%%%%%%%%%%%%%%%%%%%%%%%%%%%%%%%%%%%%%%%%%%%%%%%%%%%%%%%%%%%%%%%%%%%%%%%%%%%%%%%%%%%%%%%%%%%%%%%%%%%%%%%%%%%%%%%%%%%%%%%%%%%%%%%%%%%%%%%%%%%%%%
\section{Preliminaries}
In this section, after establishing terminology and notation, we collect some basic facts and definitions that are needed them in the sequel. 
~\\~\\
We will use the standard terminology and notation of \cite{MV} and \cite{K}. Throughout the article, $E$ will denote a nuclear Fr\'echet space with an increasing sequence of Hilbertian seminorms $\left(\left\|.\right\|_{k}\right)_{k\in \mathbb{N}}$ and the local Hilbert spaces corresponding to the norm $\left\|\cdot\right\|_{k}$ will be denoted by $E_{k}$. 
\par
For a Fr\'echet space $E$, $\mathcal{U}\left(E \right)$, $\mathcal{B}\left(E \right)$ will denote the class of all neighborhoods of zero in $E$ and  the class of all bounded sets in $E$, respectively. If $U$ and $V$ are absolutely convex sets of $E$ and $U$ absorbs $V$, that is $V\subseteq CU$ for some $C>0$, and $L$ is a subspace of $E$, then we set;
$$\delta\left(V, U, L\right)=\inf\left\{t>0: V\subseteq tU+L\right\}.$$
The $n^{th}$ \textit{Kolmogorov diameter} of $V$ with respect to $U$ is defined as;
$$d_{n}\left(V,U\right)=\inf\left\{\delta\left(V, U, L\right): \dim L\leq n\right\}\hspace{0.1in}n=0,1,....$$
and \textit{the diametral dimension} of $E$ is defined as;
\begin{equation*}
\begin{split}
\Delta\left(E\right)=& \left\{\left(t_{n}\right)_{n\in \mathbb{N}}: \forall\hspace{0.05in}U\in \mathcal{U}\left(E\right) \hspace{0.05in} \exists \hspace{0.05in}V\in\mathcal{U}\left(E\right) \hspace{0.05in}\lim_{n\rightarrow\infty} t_{n}d_{n}\left(V,U\right)=0\right\}\\ =&\displaystyle \bigcap_{{U\in \mathcal{U}\left(E \right)}}\hspace{0.05in}\bigcup_{V\in \mathcal{U}\left(E \right)}\Delta\left( V, U\right) 
\end{split}
\end{equation*}
where $\displaystyle \Delta\left( V, U\right)=\left\lbrace \left(t_{n} \right)_{n\in \mathbb{N}}: \lim_{n\rightarrow\infty}t_{n}d_{n}\left( V,U\right)=0 \right\rbrace$. 
\\Let $U_{1}\supset U_{2}\supset \cdots \supset U_{p}\supset\cdots$ be a base of neighborhoods of Fr\'echet space E. Diametral dimension can be represented as
 $$\Delta\left(E \right)= \left\{\left(t_{n}\right)_{n\in \mathbb{N}}: \forall p\in \mathbb{N} \hspace{0.1in} \exists \hspace{0.05in} q>p \hspace{0.1in}\lim_{n\rightarrow\infty} t_{n}d_{n}\left(U_{q},U_{p}\right)=0\right\}.$$
\textit{The approximate diametral dimension} of a Fr\'echet space $E$ is defined as;
\begin{equation*}
\begin{split}
\delta\left(E \right)&= \left\{\left(t_{n}\right)_{n\in \mathbb{N}}: \forall\hspace{0.05in}U\in \mathcal{U}\left(E\right) \hspace{0.05in} \forall \hspace{0.05in}B\in\mathcal{B}\left(E\right) \hspace{0.1in}\lim_{n\rightarrow\infty}{ t_{n}\over d_{n}\left(B,U\right)}=0\right\}
\\ &=\bigcup_{U\in \mathcal{U}\left(E \right)}\bigcup_{B\in \mathcal{B}\left(E \right)} \delta\left( B, U\right) 
\end{split}
\end{equation*}
where $\displaystyle\delta\left( B, U\right)=\left\lbrace \left( t_{n}\right)_{n\in \mathbb{N}}: \lim_{n\rightarrow\infty}{t_{n}\over d_{n}\left( B,U\right) }=0 \right\rbrace $. It follows from Proposition 6.6.5 of \cite{R} that for a Fr\'echet space $E$ with the base of neighborhoods $U_{1}\supset U_{2}\supset \cdots \supset U_{p}\supset\cdots$, the approximate diametral dimension can be represented as;
$$\delta\left(E \right)= \left\{\left(t_{n}\right)_{n\in \mathbb{N}}: \exists p\in \mathbb{N} \hspace{0.05in} \forall \hspace{0.05in} q>p \hspace{0.05in}\lim_{n\rightarrow\infty}{ t_{n}\over d_{n}\left(U_{q},U_{p}\right)}=0\right\}.$$
The concept of the approximative dimension of a linear metric space which is based on $\varepsilon$-capacity of compact sets in the space was introduced by Kolmogorov and Pelcyznski, see also \cite{Kol}, \cite{P} and \cite{R}. The relation between invariants introduced above and $\varepsilon$-capacity of compact sets in the space was discovered by Mityagin, \cite{M1} and \cite{M2}. Among other thing, Mityagin conducted a detailed study of these invariants and used them characterize nuclear locally convex space. The concept of approximate diametral dimension as stated above was given and studied by Bessaga, Pelczynski and Rolewicz, \cite{BRP}.  
\\Demeulenaere et al.\cite{L} showed that the diametral dimension of a nuclear Fr\'echet space can also be represented as;
$$\Delta\left(E\right)=\left\{\left(t_{n}\right)_{n\in \mathbb{N}}: \forall\hspace{0.05in}p\in \mathbb{N} \hspace{0.1in} \exists \hspace{0.05in}q> p \hspace{0.1in}\sup_{n\in \mathbb{N}} \left|t_{n}\right|d_{n}\left(U_{q},U_{p}\right)<+\infty \right\}.$$
\par Let $E$ and $G$ be two Fr\'echet spaces and $U$ and $V$ be absolutely convex two subsets of space $E$ such that $V\subseteq r U$ for some $r>0$. If there is a linear map $T: E \rightarrow G$, then for all $n\in \mathbb{N}$
$$d_{n}\left(T\left(V\right), T\left(U\right)\right)\leq d_{n}\left(V,U\right)$$
holds and so it follows that if $F$ is a subspace or a quotient of $E$, then $\Delta\left(E\right)\subseteq \Delta \left(F\right)$ and $\delta\left(F\right)\subseteq \delta\left(E\right)$. Hence diametral dimension and approximate diametral dimension are invariant under isomorphism, in other words, these are linear topological invariants. For the proof of these and for additional properties of the diametral dimension$\slash$approximate diametral dimension, we refeer the reader to \cite{BRP}, \cite{M2},\cite{Pie}(Chapter 9), \cite{R},(Chapter 6.5 and 6.6) and  \cite{T7} . 
~\\~\\
\textbf{The properties of the canonical topology on diametral dimension of a nuclear Fr\'echet space:}
~\\ Let $E$ be a nuclear Fr\'echet space. Then the diametral dimension
\begin{equation*}
\begin{split}
\Delta\left(E\right)=& \left\{\left(t_{n}\right)_{n\in \mathbb{N}}: \forall\hspace{0.05in}p\in \mathbb{N} \hspace{0.05in} \exists q>p \hspace{0.05in}\lim_{n\rightarrow\infty} t_{n}d_{n}\left(U_{q},U_{p}\right)=0\right\}\\ =& \bigcap_{p\in \mathbb{N}}\hspace{0.05in}\bigcup_{q>p}\Delta\left( U_{q}, U_{p}\right) 
\end{split}
\end{equation*}
is the projective limit of inductive limits of Banach spaces $\displaystyle \Delta\left( U_{q}, U_{p}\right)$. Hence $\Delta (E)$ is a topological space with respect to that topology which will be called as \textit{the canonical topology}. Furthermore, $\Delta (E)$ can be considered as a weighted PLB-spaces of continuous functions.
~\\ The topological properties of weighted PLB-spaces of continuous functions were studied in \cite{ABB}. In particular, the following theorem gives an information about the canonical topology of diametral dimension $\Delta\left(E\right)$ and it is a direct consequence of Theorem 3.7 of \cite{ABB}. 
\begin{thm}\label{t1} Let $E$ be a Fr\'echet space. The following conditions are equivalent:
\begin{enumerate}
	\item $\Delta\left(E\right)$ is ultrabornological with respect to the canonical topology.
	\item $\Delta\left(E\right)$ is barrelled with respect to the canonical topology.
	\item $\Delta\left(E\right)$ satisfies condition (wQ):
	$$\displaystyle \forall N\hspace{0.075in}\exists\hspace{0.075in}M, n\hspace{0.075in}\forall K, m, \hspace{0.075in}\exists\hspace{0.075in} k, S>0:\hspace{0.075in} \min\left(d_{n}\left(U_{n}, U_{N}\right), d_{n}\left(U_{k},U_{K}\right)\right)\leq S d_{n}\left(U_{m}, U_{M}\right)\hspace{0.25in} \forall n\in \mathbb{N}.$$ 
\end{enumerate}
\end{thm}
We will use this theorem in the fourth section.
\\
 Power series spaces form an important family of Fr\'echet spaces and they play an significant role in this article. Let $\alpha=\left(\alpha_{n}\right)_{n\in \mathbb{N}}$ be a non-negative increasing sequence with $\displaystyle \lim_{n\rightarrow \infty} \alpha_{n}=+\infty$.  Thoughout this article, all power series spaces are assumed to be nuclear. Recall that a power series space of finite type is defined by
$$\displaystyle \Lambda_{1}\left(\alpha\right):=\left\{x=\left(x_{n}\right)_{n\in \mathbb{N}}: \left\|x\right\|_{k}:=\sup_{n\in \mathbb{N}}\left|x_{n}\right|e^{{-{1\over k}\alpha_{n}}}<+\infty \textnormal{ for all } k\in \mathbb{N}\right\}$$
and a power series space of infinite type is defined by
$$\displaystyle \Lambda_{\infty}\left(\alpha\right):=\left\{x=\left(x_{n}\right)_{n\in \mathbb{N}}: \left\|x\right\|_{k}:=\sup_{n\in\mathbb{N}}\left|x_{n}\right|e^{k\alpha_{n}}<+\infty \textnormal{ for all } k\in \mathbb{N}\right\}.$$
Power series spaces are actually Fr\'echet spaces equipped with the seminorms $\left(\left\|.\right\|_{k}\right)_{k\in \mathbb{N}}$. 
Diametral dimension and approximate diametral dimension of power series spaces are
$$\Delta\left(\Lambda_{1}\left(\alpha\right)\right)=\Lambda_{1}\left(\alpha\right), \hspace{.2in} \Delta\left(\Lambda_{\infty}\left(\alpha\right)\right)=\Lambda_{\infty}\left(\alpha\right)^{\prime},\hspace{0.2in}\delta\left(\Lambda_{1}\left(\alpha\right)\right)=\Lambda_{1}\left(\alpha\right)^{\prime}, \hspace{0.2in}\textnormal{and}\hspace{0.2in} \delta\left(\Lambda_{\infty}\left(\alpha\right)\right)=\Lambda_{\infty}\left(\alpha\right),$$
see \cite{BRP} and \cite{M2}.
\par Other linear topological invariants that are used in this article are $\underline{DN}$ and $\Omega$, see \cite{MV} and references therein. 
\begin{defnt} A nuclear Fr\'echet space  $E$ is said to have the property \underline{DN} and $\Omega$ when the following conditions hold: ~\\
$\underline{\left(DN\right):}$ $\hspace{0.1in}$ There exists a $p\in \mathbb{N}$ such that for each $k>p$, an $n>k$, $0<\tau<1$ and a $C>0$ exist with
$$\left\|x\right\|_{k}\leq C\left\|x\right\|^{1-\tau}_{p}\left\|x\right\|^{\tau}_{n}\hspace{0.15in}\textnormal{for all }x\in E.$$
$\left(\Omega\right):$ $\hspace{0.15in}$ For each $p\in \mathbb{N}$, there exists a $q>p$ such that for every $k>q$ there exists a $0<\theta<1$ and a $C>0$ with
$$\left\|y\right\|^{*}_{q}\leq C{\left\|y\right\|^{*}}^{1-\theta}_{p}{\left\|y\right\|^{*}}^{\theta}_{k}\hspace{0.15in}\textnormal{ for all }\hspace{0.05in} y\in E^{'},$$
where
$$\displaystyle \left\|y\right\|^{*}_{k}:=\sup\left\{\left|y\left(x\right)\right|: \left\|x\right\|_{k}\leq 1\right\}\in \mathbb{R}\cup \left\{+\infty\right\}$$
is called the gauge functional of $U^{\circ}_{k}$ for $U_{k}= \left\{x\in E: \left\|x\right\|_{k}\leq 1\right\}$.
\end{defnt}
~\\  We end this section by recalling the following result which gives a relation between diametral dimension/approximate diametral dimension of a nuclear Fr\'echet spaces with the properties $\underline{DN}$, $\Omega$ and that of a power series spaces $\Lambda_{1}\left(\varepsilon\right)$ and $\Lambda_{\infty}\left(\varepsilon\right)$ for some special exponent sequence $\varepsilon$.
\begin{prop}\label{p1}$\left[\right.$Proposition 1.1, \cite{AKT2} $\left.\right]$ Let $E$ be a nuclear Fr\'echet space with the properties $\underline{DN}$ and $\Omega$. There exists an exponent sequence (unique up to equivalence) $\left(\varepsilon_{n}\right)$ satisfying:
$$\Delta\left(\Lambda_{1}\left(\varepsilon\right)\right)\subseteq\Delta\left(E\right)\subseteq\Delta\left(\Lambda_{\infty}\left(\varepsilon\right)\right).$$
Furthermore, $\Lambda_{1}\left(\alpha\right)\subseteq\Delta\left(E\right)$ implies $\Lambda_{1}\left(\alpha\right)\subseteq \Lambda_{1}\left(\varepsilon\right)$ and $\Delta\left(E\right)\subseteq \Lambda^{\prime}_{\infty}\left(\alpha\right)$ implies $\Lambda^{\prime}_{\infty}\left(\varepsilon\right)\subseteq\Lambda^{'}_{\infty}\left(\alpha\right)$.
\end{prop}
The sequence $\varepsilon$ was called the \textit{associated exponent sequence} of $E$ in \cite{AKT2}. The exponent sequence $\varepsilon$ associated to $E$ contains some information about the structure of $E$. We note that $\Lambda_{\infty} (\varepsilon)$ is always nuclear provided $E$ is nuclear, but it may happen that $\Lambda_{1} (\varepsilon)$ is not nuclear. Throughout this article, we assume $\Lambda_{1} (\varepsilon)$ is nuclear for associated exponent sequence $\varepsilon$ of a nuclear Fr\'echet space $E$. In the proof of the above proposition, Aytuna et al. showed that there is an exponent sequence (unique up to equivalence) $\left(\varepsilon_{n}\right)$ such that for each $p\in \mathbb{N}$ and $q>p$, there exist $C_{1}, C_{2}>0$ and $a_{1}, a_{2}>0$ satisfying $C_{1}e^{{-a_{1}\varepsilon_{n}}}\leq d_{n}\left(U_{q},U_{p}\right)\leq C_{2}e^{{-a_{2}\varepsilon_{n}}}$ for all $n\in \mathbb{N}$. From this inequality, it follows $$\delta\left(\Lambda_{\infty}\left(\varepsilon\right)\right)\subseteq\delta\left(E\right)\subseteq\delta\left(\Lambda_{1}\left(\varepsilon\right)\right).$$

%%%%%%%%%%%%%%%%%%%%%%%%%%%%%%%%%%%%%%%%%%%%%%%%%%%%%%%%%%%%%%%%%%

\section{Results in the infinite case}

The main result in this section is the following theorem which shows that Question \ref{q1} has an affirmative answer when the power series space is of infinite type.
\begin{thm}\label{t2} Let $E$ be a nuclear Fr\'echet space with properties \underline{DN} and $\Omega$ and $\varepsilon=\left(\varepsilon_{n}\right)_{n\in \mathbb{N}}$ be the associated exponent sequence of $E$. Then $\Delta\left(E\right)= \Delta\left(\Lambda_{\infty}\left(\varepsilon\right)\right)$ if and only if $\delta\left(E\right)= \delta\left(\Lambda_{\infty}\left(\varepsilon\right)\right).$
\end{thm}

We first need the following lemma for the proof of Theorem \ref{t2}. In \cite[Cor.\ 1.10]{A}, Aytuna proved that for a nuclear Fr\'echet space $E$ with the properties $\underline{DN}$, $\Omega$ and associated exponent sequence $\varepsilon$
\begin{equation}\label{e1}\delta\left(E\right)= \delta\left(\Lambda_{1}\left(\varepsilon\right)\right) \hspace{0.25in}\Leftrightarrow \hspace{0.25in} \inf_{p}\sup_{q\geq p} \limsup_{n\in \mathbb{N}}{\varepsilon_{n}\left(p,q\right)\over \varepsilon_{n}}=0\end{equation}
where $\varepsilon_{n}\left(p,q\right)=-\log{d_{n}\left(U_{q},U_{p}\right)}$. 
~\\ The same characterization can be given for infinite type power series spaces as follows:

%We start investigating the conditions which ensure that the approximate diametral dimension of a nuclear Fr\'echet space $E$ coincide with the approximate diametral dimension of a power series space.
%We will use the notation $\varepsilon_{n}\left(p,q\right)$ to denote
%$$\hspace{0.5in} \varepsilon_{n}\left(p,q\right)=-\ln d_{n}\left(U_{q}, U_{p}\right)\hspace{0.5in} \forall n\in \mathbb{N}.$$

\begin{lemma}\label{l1} Let $E$ be a nuclear Fr\'echet space with properties \underline{DN} and $\Omega$ and $\varepsilon=\left(\varepsilon_{n}\right)_{n\in \mathbb{N}}$ be the associated exponent sequence of $E$. Then
$$\displaystyle \delta\left(E\right)=\delta\left(\Lambda_{\infty}\left(\varepsilon\right)\right)\hspace{0.1in}\Leftrightarrow\hspace{0.1in} \inf_{p\in \mathbb{N}}\sup_{q>p}\liminf_{n\in \mathbb{N}}{\varepsilon_{n}\left(p,q\right)\over \varepsilon_{n}}=+\infty$$
where $\varepsilon_{n}\left(p,q\right)=-\log{d_{n}\left(U_{q},U_{p}\right)}$.
\end{lemma}
\begin{proof} Approximate diametral dimension $\delta\left(E\right)$ can be written as $$\displaystyle \delta\left(E\right)= \bigcup_{p}\bigcap_{q\geq p}\delta_{pq}$$ where $\displaystyle \delta_{pq}=\left\{ \left(t_{n}\right)_{n\in \mathbb{N}}: \sup_{n\in \mathbb{N}}{\left|t_{n}\right|\over d_{n}\left(U_{q},U_{p}\right)}<+\infty\right\}$ is a Banach space with norms $\displaystyle \left|t_{n}\right|_{pq}=\sup_{n\in \mathbb{N}}{\left|t_{n}\right|\over d_{n}\left(U_{q},U_{p}\right)}$. Namely, approximate diametral dimension can be equipped with the topological inductive limit of Fr\'echet spaces. Then, the approximate diametral dimension with this topology is barrelled. On the other hand, the inclusion $\delta\left(E\right)\subseteq \delta\left(\Lambda_{\infty}\left(\varepsilon\right)\right)=\Lambda_{\infty}\left(\varepsilon\right)$ gives us that the identity mapping $i:\delta\left(E\right)\rightarrow\Lambda_{\infty}\left(\varepsilon\right)$ has a closed graph. Since $\delta\left(E\right)$ is barrelled, by using Theorem 5 of \cite{T}, we conclude that the identity mapping is continuous. Therefore,
~\\ \begin{equation*}
\begin{split}
& \delta\left(E\right)=\bigcup_{p}\bigcap_{q\geq p}\delta_{pq}\hookrightarrow \Lambda_{\infty}\left(\varepsilon\right) \hspace{0.1in}\textnormal{is continuous}\hspace{0.1in} \Leftrightarrow \hspace{0.1in}\forall p \hspace{0.15in}\bigcap_{q\geq p}\delta_{pq}\hookrightarrow \Lambda_{\infty}\left(\varepsilon \right) \hspace{0.1in}\textnormal{is continuous} \vspace{1in}
\\ & \Leftrightarrow \hspace{0.1in}\forall p \hspace{0.1in}\forall R>1 \hspace{0.1in}\exists q\geq p,\hspace{0.1in} C>0 \hspace{0.1in}\sup_{n\in \mathbb{N}}\left|t_{n}\right|R^{\varepsilon_{n}}\leq C\sup_{n\in \mathbb{N}}{\left|t_{n}\right|\over d_{n}\left(U_{q},U_{p}\right)}\hspace{0.1in} \forall \left(t_{n}\right)\in \delta\left(E\right) \vspace{2in}
\\  & \Leftrightarrow \hspace{0.1in}\forall p \hspace{0.1in}\forall R>1 \hspace{0.1in}\exists q\geq p, C>0\hspace{0.1in} R^{\varepsilon_{n}}\leq {C\over d_{n}\left(U_{q},U_{p}\right)} \hspace{0.1in}\forall n\in \mathbb{N} \vspace{2in}
\\  & \Leftrightarrow \hspace{0.1in} \forall p \hspace{0.1in}\forall R>1 \hspace{0.1in}\ln{R}\leq \sup_{q\geq p}\liminf_{n\in \mathbb{N}} {\varepsilon_{n}\left(p,q\right)\over \varepsilon_{n}}\vspace{2in}
\\ & \Leftrightarrow \hspace{0.1in} \forall p \hspace{0.1in} \sup_{q\geq p}\liminf_{n\in \mathbb{N}} {\varepsilon_{n}\left(p,q\right)\over \varepsilon_{n}}=+\infty \hspace{0.2in} \Leftrightarrow \hspace{0.2in} \inf_{p\in \mathbb{N}}\sup_{q\geq p}\liminf_{n\in \mathbb{N}} {\varepsilon_{n}\left(p,q\right)\over \varepsilon_{n}}=+\infty.
\end{split}
\end{equation*}
Now since $\delta\left(E\right)\supseteq \delta\left(\Lambda_{\infty}\left(\varepsilon\right)\right)$ always holds for the associated exponent sequence $\varepsilon$ of $E$, we have
$$\delta\left(E\right)= \delta\left(\Lambda_{\infty}\left(\varepsilon\right)\right)\hspace{0.25in} \Leftrightarrow\hspace{0.25in}\displaystyle \inf_{p\in \mathbb{N}}\sup_{q\geq p}\liminf_{n\in \mathbb{N}} {\varepsilon_{n}\left(p,q\right)\over \varepsilon_{n}}=+\infty,$$
as desired.
\end{proof}

\begin{proof*} For the proof of necessity part, assume that $\delta\left(E\right)= \delta\left(\Lambda_{\infty}\left(\varepsilon\right)\right)$. By Lemma \ref{l1}, 
$\displaystyle \inf_{p\in \mathbb{N}}\sup_{q>p}\liminf_{n\in \mathbb{N}}{\varepsilon_{n}\left(p,q\right)\over \varepsilon_{n}}=+\infty$. Then we have 
$$\forall p \hspace{0.1in} \forall M>0 \hspace{0.1in} \exists q\geq p \hspace{0.1in}\liminf_{n\in \mathbb{N}}{\varepsilon_{n}\left(p,q\right)\over \varepsilon_{n}}\geq M$$
and
$$ \hspace{1.65in}\forall p \hspace{0.1in} \forall M>0 \hspace{0.1in} \exists q\geq p \hspace{0.1in} d_{n}\left(U_{q}, U_{p}\right)\leq e^{{-M\varepsilon_{n}}} \hspace{1.5in}\forall n\in \mathbb{N} \hspace{0.05in}\left(*\right).$$ 
Now, if we take $\left(x_{n}\right)_{n\in \mathbb{N}}\in \Delta\left(\Lambda_{\infty}\left(\varepsilon\right)\right)$, then there exists a $S>0$ such that $\displaystyle \sup_{n\in \mathbb{N}} \left|x_{n}\right|e^{{-S\varepsilon_{n}}}<+\infty$ which means that there exists a $C>0$ such that for every $n\in \mathbb{N}$
$$\left|x_{n}\right|\leq Ce^{{S\varepsilon_{n}}}.$$
Now, for a fixed $p$ and the number $S$, from $\left(*\right)$ we can find a $q\geq p$ such that for every $n\in \mathbb{N}$
$$\left|x_{n}\right|d_{n}\left(U_{q},U_{p}\right)\leq Ce^{{S\varepsilon_{n}}}e^{{-S\varepsilon_{n}}}=C.$$
Then, $\left(x_{n}\right)_{n\in \mathbb{N}}\in \Delta\left(E\right)$ and so $\Delta\left(\Lambda_{\infty}\left(\varepsilon\right)\right)\subseteq \Delta\left(E\right)$. But then since we always have $\Delta\left(E\right)\subseteq \Delta\left(\Lambda_{\infty}\left(\varepsilon\right)\right)$, we obtain $\Delta\left(E\right)=\Delta\left(\Lambda_{\infty}\left(\varepsilon\right)\right).$ 
~\\  To prove the sufficiency part, assume $\Delta\left(E\right)=\Delta\left(\Lambda_{\infty}\left(\varepsilon\right)\right)$ and $\delta\left(E\right)\neq\delta\left(\Lambda_{\infty}\left(\varepsilon\right)\right).$ 
\begin{equation*}
\begin{split}
& \delta\left(E\right)\neq\delta\left(\Lambda_{\infty}\left(\varepsilon\right)\right) \hspace{0.1in} \Leftrightarrow \hspace{0.1in} \exists p \hspace{0.05in} \sup_{q\geq p}\liminf_{n\in \mathbb{N}}{\varepsilon_{n}\left(p,q\right)\over \varepsilon_{n}}<+\infty 
\\ & \Leftrightarrow \exists p \hspace{0.1in} \exists M>0 \hspace{0.1in} \sup_{q\geq p}\liminf_{n\in \mathbb{N}}{\varepsilon_{n}\left(p,q\right)\over \varepsilon_{n}}\leq M
\\ & \Leftrightarrow\exists p \hspace{0.1in} \exists M>0 \hspace{0.1in} \forall q\geq p \hspace{0.1in}\liminf_{n\in \mathbb{N}}{\varepsilon_{n}\left(p,q\right)\over \varepsilon_{n}}\leq M
\\ & \Leftrightarrow \exists p \hspace{0.1in} \exists M>0 \hspace{0.1in} \forall q\geq p \hspace{0.1in} \exists I_{q}\subseteq \mathbb{N}\hspace{0.1in} d_{n}\left(U_{q},U_{p}\right)\geq e^{-{M\varepsilon_{n}}}\hspace{0.35in} \forall n\in I_{q}
\end{split}
\end{equation*}
Now since $\displaystyle\Delta\left(E\right)=\Delta\left(\Lambda_{\infty}\left(\varepsilon\right)\right)= {\Lambda_{\infty}\left(\varepsilon\right)}^{\prime}=\left\{\left(x_{n}\right)_{n\in \mathbb{N}}: \exists R>0 \hspace{0.1in} \displaystyle \sup_{n\in \mathbb{N}}\left|x_{n}\right|e^{-R{\varepsilon_{n}}}<+\infty\right\}$, for every $R>0$, we have $\displaystyle e^{R{\varepsilon_{n}}}\in {\Lambda_{\infty}\left(\varepsilon\right)}^{\prime}=\Delta\left(E\right)$. Therefore, for the above $p$, we can find a $\tilde{q}>p$, such that
$$\sup_{n\in \mathbb{N}}e^{R{\varepsilon_{n}}}d_{n}\left(U_{\tilde{q}}, U_{p}\right)<+\infty.$$
Then for every $n\in I_{\tilde{q}}$, we obtain
$$e^{\left(R-M\right){\varepsilon_{n}}}\leq e^{R{\varepsilon_{n}}}d_{n}\left(U_{\tilde{q}}, U_{p}\right)\leq \sup_{n\in \mathbb{N}}e^{R{\varepsilon_{n}}}d_{n}\left(U_{\tilde{q}}, U_{p}\right)<+\infty.$$
But then if we choose $R>M$, we have a contradiction. Hence $\Delta\left(E\right)= \Delta\left(\Lambda_{\infty}\left(\varepsilon\right)\right)$ implies $\delta\left(E\right)= \delta\left(\Lambda_{\infty}\left(\varepsilon\right)\right)$, as desired.
\end{proof*}
%%%%%%%%%%%%%%%%%%%%%%%%%%%%%%%%%%%%%%%%%%%%%%%%%%%%%%%%%%%%%%%%%%

\section{Results in the finite case}
In this section, we turn our attention to the finite type power series case and, as a main result, we prove that Question \ref{q1} is true in case the nuclear Fr\'echet space $E$ has a prominent bounded subset. 
%First of all, we show that $\delta\left(E\right)= \delta\left(\Lambda_{1}\left(\varepsilon\right)\right)$ implies $\Delta\left(E\right)= \Delta\left(\Lambda_{1}\left(\varepsilon\right)\right)$ which supports Question \ref{q1} in one direction. Furthermore, we give some certain conditions on Kolmogorov diameters of $E$ for which Question \ref{q1} is verified in the other direction. We also examine the conditions on $E$ to have a prominent bounded subset. Then as a main result of this section, we prove that if $E$ is a nuclear Fr\'echet space with the properties \underline{DN} and $\Omega$ has a prominent bounded set  and $\Delta\left(E\right)=\Lambda_{1}\left(\varepsilon\right)$ if and only if $\delta\left(E\right)=\delta\left(\Lambda_{1}\left(\varepsilon\right)\right)$.
~\\ We begin this section by giving the following proposition which answers Question \ref{q1} in one direction.
\begin{prop}\label{p2} Let $E$ be a nuclear Fr\'echet space with properties \underline{DN} and $\Omega$ and $\varepsilon=\left(\varepsilon_{n}\right)_{n\in \mathbb{N}}$ be the associated exponent sequence of $E$. Then $\delta\left(E\right)= \delta\left(\Lambda_{1}\left(\varepsilon\right)\right)$ implies $\Delta\left(E\right)= \Delta\left(\Lambda_{1}\left(\varepsilon\right)\right)$.
\end{prop}
\begin{proof} Let us assume that $\delta\left(E\right)= \delta\left(\Lambda_{1}\left(\varepsilon\right)\right)$. From Corollary 1.10 of \cite{A}, we have
\begin{equation*}
\begin{split}
\delta\left(E\right)= \delta\left(\Lambda_{1}\left(\varepsilon\right)\right) \hspace{0.1in} &\Leftrightarrow \hspace{0.1in} \inf_{p}\sup_{q\geq p} \limsup_{n\in \mathbb{N}}{\varepsilon_{n}\left(p,q\right)\over \varepsilon_{n}}=0
\\ & \Leftrightarrow \forall r>0 \hspace{0.1in} \exists p \hspace{0.1in}\forall q\geq p \hspace{0.1in} \limsup_{n\in \mathbb{N}}{\varepsilon_{n}\left(p,q\right)\over \varepsilon_{n}}\leq r
\\ & \Leftrightarrow \forall r>0 \hspace{0.1in} \exists p \hspace{0.1in}\forall q\geq p \hspace{0.1in} \exists n_{0}\in \mathbb{N} \hspace{0.1in}\forall n\geq n_{0} \hspace{0.1in} d_{n}\left(U_{q},U_{p}\right)\geq e^{-r{\varepsilon_{n}}}.
\end{split}
\end{equation*}
Now, we take $\left(x_{n}\right)_{n\in \mathbb{N}}\in \Delta\left(E\right)$ and for the above $p$, we find a $\tilde{q}> p$ such that
$$\sup_{n\in \mathbb{N}} \left|x_{n}\right|d_{n}\left(U_{\tilde{q}},U_{p}\right)<+\infty$$
and from the above inequality, we obtain
$$\left|x_{n}\right|e^{{-r\varepsilon_{n}}}\leq \sup_{n\in \mathbb{N}} \left|x_{n}\right|d_{n}\left(U_{\tilde{q}},U_{p}\right)$$
for large n, this means that $\left(x_{n}\right)_{n\in \mathbb{N}}\in \Delta\left(\Lambda_{1}\left(\varepsilon\right)\right)$ and so $\Delta\left(E\right)\subseteq \Delta\left(\Lambda_{1}\left(\varepsilon\right)\right)$. But then since $\Delta\left(E\right)\supseteq \Delta\left(\Lambda_{1}\left(\varepsilon\right)\right)$, we have $\Delta\left(E\right)= \Delta\left(\Lambda_{1}\left(\varepsilon\right)\right).$
\end{proof}
\begin{thm}\label{t4} Let $E$ be a nuclear Fr\'echet space with properties \underline{DN} and $\Omega$ and $\varepsilon=\left(\varepsilon_{n}\right)_{n\in \mathbb{N}}$ be the associated exponent sequence of $E$. If $\Delta\left(E\right)$, with the canonical topology, is barrelled, then $\Delta\left(E\right)=\Delta\left(\Lambda_{1}\left(\varepsilon\right)\right)$ if and only if $\delta\left(E\right)=\delta\left(\Lambda_{1}\left(\varepsilon\right)\right)$.
\end{thm} 
\begin{proof} The proof of the necessity part follows from Proposition \ref{p2}. To prove the sufficiency part, let $\Delta\left(E\right)=\Delta\left(\Lambda_{1}\left(\varepsilon\right)\right)$ and assume that $\Delta\left(E\right)$ with the canonical topology be barrelled. Then since the convergence in $\Delta\left(E\right)$ implies the coordinate-wise convergence, the inclusion $\Delta\left(E\right)\hookrightarrow\Lambda_{1}\left(\varepsilon\right)$ has a closed graph. But then since $\Delta\left(E\right)$ is barrelled, the inclusion map $\Delta\left(E\right)\hookrightarrow\Lambda_{1}\left(\varepsilon\right)$ is continuous by Theorem 5 in \cite{T}. Taking into account that $\displaystyle \Delta\left(E\right)$ is the projective limit of inductive limits of Banach spaces $\displaystyle \bigcap_{p\in \mathbb{N}}\bigcup_{q\geq p+1} \Delta\left(U_{q},U_{p}\right)$, the continuity of the inclusion map $\displaystyle \bigcap_{p\in \mathbb{N}}\bigcup_{q\geq p+1} \Delta\left(U_{q},U_{p}\right)\hookrightarrow\Lambda_{1}\left(\varepsilon\right)$ gives us
$$\forall t>0\hspace{0.1in}\exists p\hspace{0.1in} \forall q>p \hspace{0.1in}\exists C>0 \hspace{0.1in} \forall n\in \mathbb{N} \hspace{0.2in}e^{-t\varepsilon_{n}}\leq C\hspace{0.05in}d_{n}\left( U_{q}, U_ {p}\right).$$
This implies $\displaystyle \inf_{p\in \mathbb{N}}\sup_{q>p}\limsup_{n\in \mathbb{N}}{ \varepsilon_{n} (p,q)\over \varepsilon_{n}}=0$, so $\delta\left(E\right)=\delta\left(\Lambda_{1}\left(\varepsilon\right)\right)$, as desired.
\end{proof}
 It is worth the note that, by Theorem \ref{t1}, the barrelledness of the canonical topology of $\Delta\left(E\right)$ is equivalent to the following condition (wQ):
$$\displaystyle \forall N\hspace{0.075in}\exists\hspace{0.075in}M, n\hspace{0.075in}\forall K, m, \hspace{0.075in}\exists\hspace{0.075in} k, S>0:\hspace{0.075in} \min\left(d_{n}\left(U_{n}, U_{N}\right), d_{n}\left(U_{k},U_{K}\right)\right)\leq S d_{n}\left(U_{m}, U_{M}\right)\hspace{0.25in} \forall n\in \mathbb{N}.$$ 
But determining the barrelledness of $\Delta (E)$ is not easy, in practice. In the following proposition, by posing a condition
\\ \textbf{Condition} $\boldsymbol{\mathbb{A}}$: $\hspace{0.1in}\forall$
p , $\hspace{0.1in}\forall$ $q>p$, $\hspace{0.1in}\exists$ $s>q$, $\hspace{0.1in}\forall k>s$, $\hspace{0.1in}\exists C>0$ $\hspace{0.15in}d_{n}\left(U_{q}, U_{p}\right)\leq C d_{n}\left(U_{k}, U_{s}\right)$ $\hspace{0.15in}\forall n\in \mathbb{N}$.
\\on diameters, we eliminate the barrelledness condition of Theorem 4.2.

\begin{prop}\label{p3} Let $E$ be a nuclear Fr\'echet space with the properties \underline{DN} and $\Omega$ and $\varepsilon$ be the associated exponent sequence of $E$. If $E$ satisfies the condition $\mathbb{A}$
and $\Delta\left(E\right)=\Delta\left(\Lambda_{1}\left(\varepsilon\right)\right)$, then $\delta\left(E\right)=\delta\left(\Lambda_{1}\left(\varepsilon\right)\right)$.
\end{prop}
\begin{proof} Suppose that $E$ satisfies the condition $\mathbb{A}$ and $\Delta\left(E\right)=\Delta\left(\Lambda_{1}\left(\varepsilon\right)\right)$. If $\delta\left( E\right)\neq \delta\left(\Lambda_{1}\left( \varepsilon\right) \right)$, then from Corollary 1.10 of \cite{A} we have the following condition:
\begin{equation}\hspace{1in}\exists M>0 \hspace{0.1in} \forall p \hspace{0.1in} \exists q_{p}>p,\hspace{0.05in} I_{p}\subseteq \mathbb{N} \hspace{0.15in} d_{n}\left(U_{q}, U_{p} \right)<e^{-M\varepsilon_{n}}\hspace{0.2in} \forall n\in I_{p} \end{equation}
For p=1, there exists a number $q_{1}$ and an infinite subset $I_{1}$ so that for all $n\in I_{1}$
$$d_{n}\left(U_{q_{1}}, U_{p} \right)<e^{-M\varepsilon_{n}},$$
and so it follows from the condition $\mathbb{A}$ that we have a number  $q_{2}$ such that for all $k\geq q_{2}$ there exists a $C>0$ so that for all $n\in \mathbb{N}$
$$d_{n}\left(U_{q_{1}}, U_{1}\right)\leq C d_{n}\left(U_{k}, U_{q_{2}}\right)$$ holds. Then, from the inequality 4.1, there exists a number $q_{3}$ and an infinite subset $I_{2}$ so that for all $n\in I_{2}$
$$d_{n}\left(U_{q_{3}}, U_{q_{2}} \right)<e^{-M\varepsilon_{n}}.$$
It follows that there exists a $C_{1}>0$ so that for all $n\in \mathbb{N}$
$$d_{n}\left(U_{q_{1}}, U_{1}\right)\leq C_{1} d_{n}\left(U_{q_{3}}, U_{q_{2}}\right)$$
holds. %Again from $\mathbb{A}$, there exists a number $q_{4}$ such that for all $k\geq q_{4}$ there exists a $C>0$ so that for all $n\in \mathbb{N}$
%$$d_{n}\left(U_{q_{3}}, U_{q_{2}}\right)\leq C d_{n}\left(U_{k}, U_{q_{4}}\right)$$
%holds. Again from the inequality 3.1, there exists a number $q_{5}$ and a infinite subset $I_{3}$ so that for all $n\in I_{2}$
%$$d_{n}\left(U_{q_{5}}, U_{q_{4}} \right)<e^{-M\varepsilon_{n}}.$$
Now applying the same process for $q_{2}$ and $q_{3}$, we can find $q_{4}$, $q_{5}$ and $C_{2}>0$ such that
$$d_{n}\left(U_{q_{3}}, U_{q_{2}}\right)\leq C_{2} d_{n}\left(U_{q_{5}}, U_{q_{4}}\right),$$
for all $n\in \mathbb{N}$. Continuing in this way, we can find the sequences $\left\{q_{k}\right\}_{k\in \mathbb{N}}$ and $\left\{C_{k}\right\}_{k\in \mathbb{N}}$ satisfying 
\begin{equation}
\displaystyle d_{n}\left(U_{q_{{1}}}, U_{1}\right)\leq C_{1} d_{n}\left(U_{q_{{3}}}, U_{q_{{2}}}\right)\leq C_{2} d_{n}\left(U_{q_{{5}}}, U_{q_{{4}}}\right)\leq \cdots \leq C_{k} d_{n}\left(U_{q_{{2k+1}}}, U_{q_{{2k}}}\right)\leq \cdots 
\end{equation}
for all $n\in \mathbb{N}$. Moreover, for each $k\in \mathbb{N}$, there exists a $I_{k}\subseteq \mathbb{N}$ so that 
\begin{equation}
d_{n}\left(U_{q_{{2k+1}}}, U_{q_{{2k}}}\right) <e^{-M\varepsilon_{n}}
\end{equation}
for all $n\in I_{k}$.
~\\ Now, for each $k\in \mathbb{N}$, we define 
$$B_{k}=\left\lbrace x=\left(x_{n} \right): \sup_{n\in \mathbb{N}} C_{k}\vert x_{n}\vert d_{n}\left( U_{q_{{2k+1}}}, U_{q_{{2k}}}\right)<+\infty  \right\rbrace,$$
where $B_{k}$ is a Banach space under the norm $\displaystyle \Vert x\Vert_{k}=\sup_{n\in \mathbb{N}}C_{k}\vert x_{n}\vert d_{n}\left( U_{q_{{2k+1}}}, U_{q_{{2k}}}\right)$ for all $k\in \mathbb{N}$. By the inequality 4.2, we have  $B_{k+1}\subseteq B_{k}$ and $\Vert\cdot\Vert_{k}\leq \Vert\cdot\Vert_{k+1}$ for all $k\in \mathbb{N}$. Since $\left(q_{k}\right)_{k\in \mathbb{N}}$ is strictly increasing and unbounded, for all $p\in \mathbb{N}$, there exists a $k_{0}\in \mathbb{N}$ such that $\displaystyle q_{2k_{0}}>p$ and this gives us $\displaystyle U_{q_{{2k_{0}}}}\subseteq U_{p}$. For all $n\in \mathbb{N}$
$$ d_{n}\left(U_{q_{{2k_{0}+1}}}, U_{p}\right)\leq d_{n}\left(U_{q_{{2k_{0}+1}}}, U_{q_{{2k_{0}}}}\right),$$
which means that $\displaystyle \bigcap_{k}B_{k}\subseteq \Delta\left( E\right)$. Moreover, the equality $\Delta\left( X\right)= \Lambda_{1}\left( \varepsilon\right)$ yields a continuous imbedding of the projective limit  $\displaystyle \bigcap_{k}B_{k}$ into $ \Lambda_{1}\left( \varepsilon\right)$. Then since $\displaystyle\bigcap_{k}B_{k}$ and $\Lambda_{1}\left( \varepsilon\right)$ are Fr\'echet spaces and the imbedding map has a closed graph, by Theorem 5 in \cite{T}, this map is continuous and so
\begin{equation*}
\begin{split}
\bigcap_{k}B_{k}\hookrightarrow \Delta\left(\Lambda_{1}\left( \varepsilon\right) \right)\hspace{0.05in} \textnormal{is continuous} \hspace{0.05in}&\Leftrightarrow \hspace{0.1in}  \forall t>0 \hspace{0.1in} \exists k, C \hspace{0.15in} \sup_{n}\vert x_{n}\vert e^{-t\varepsilon_{n}}\leq C\hspace{0.05in}\sup_{n\in \mathbb{N}} \vert x_{n}\vert d_{n}\left(U_{q_{{2k+1}}}, U_{q_{{2k}}}\right)\\ & \hspace{0.35in}\forall \left( x_{n}\right)\in \bigcap_{p}B_{p} \\
& \Leftrightarrow \hspace{0.15in}  \forall t>0 \hspace{0.15in} \exists k, C \hspace{0.15in} \forall n\in \mathbb{N}\hspace{0.15in}  e^{-t\alpha_{n}}\leq C\hspace{0.05in} d_{n}\left(U_{q_{{2k+1}}}, U_{q_{{2k}}}\right).
\end{split}
\end{equation*}
But, this is contradictory to the inequality 4.3. Therefore, $\delta\left( E\right)=\delta\left(\Lambda_{1}\left( \varepsilon\right) \right)$ holds when $\Delta\left( E\right)= \Delta\left(\Lambda_{1}\left( \varepsilon\right) \right)$ and the condition $\mathbb{A}$ holds.
\end{proof}
There could be other diameter conditions as above which yields the same conclusion in Proposition 4.3. For example, by introducing
~\\
\textbf{Condition} $\boldsymbol{\mathbb{B}}$: $\hspace{0.2in}\forall$ $p$ $\hspace{0.1in}\forall$ $q_{1}$, $q_{2}$,..., $q_{p}$, $\hspace{0.1in}\exists$ $1\leq s\leq p$, $\hspace{0.1in}\exists C>0$ $\displaystyle \hspace{0.15in}\max_{1\leq i \leq q} d_{n}\left(U_{q_{i}}, U_{i}\right)\leq C d_{n}\left(U_{q_{s}}, U_{s}\right)$ $\hspace{0.3in}\forall n\in \mathbb{N}$.
~\\ we have

\begin{prop}\label{p31} Let $E$ be a nuclear Fr\'echet space with the properties \underline{DN} and $\Omega$ and $\varepsilon$ be associated exponent sequence of $E$. If $E$ satisfies the condition $\mathbb{B}$ and $\Delta\left(E\right)=\Delta\left(\Lambda_{1}\left(\varepsilon\right)\right)$, then $\delta\left(E\right)=\delta\left(\Lambda_{1}\left(\varepsilon\right)\right)$.
\end{prop}
The proof is similar to Proposition 4.3 except that the projective limit will be replaced by $\displaystyle \bigcap_{k}D_{k}$ where $\displaystyle D_{k}=\left\{ x=\left(x_{n} \right): \sup_{n\in \mathbb{N}} \vert x_{n}\vert \max _{1\leq i\leq p}d_{n}\left( U_{q_{i}}, U_{i}\right)<+\infty  \right\rbrace$.
~\\~\\
For some Fr\'echet spaces, one can obtain the diametral dimension by using a single bounded subset:
~\\ Terzio\u{g}lu \cite{T2} introduced an absolutely convex bounded subset $B$ of a Fr\'echet space E as  \textit{prominent bounded set} in case $\displaystyle \lim_{n\rightarrow +\infty} x_{n}d_{n}\left(B, U_{p}\right)=0$ for every $p$ implies $\left(x_{n}\right)\in \Delta\left(E\right)$. If $E$ has a prominent bounded set $B$, then $$\Delta\left(E\right)= \left\{\left(x_{n}\right)_{n\in \mathbb{N}}: \forall p, \lim_{n\rightarrow+\infty} x_{n}d_{n}\left(B,U_{p}\right)=0 \right\}.$$
In this case one can introduce a natural Fr\'echet space topology on $\Delta\left(E\right)$. Terzio\u{g}lu also gave a necessary and sufficient condition for a bounded subset to be prominent [Proposition 3, \cite{T6}], namely, $B$ is a prominent set if and only if for each p there is a q and $C>0$ such that
$$d_{n}\left(U_{q}, U_{p}\right)\leq Cd_{n}\left(B, U_{q}\right)$$
hold for all $n\in \mathbb{N}$.
~\\In the following proposition, we prove that having a prominent bounded subset is closely related to Bessaga's basis free version of Dragilev condition $d_{2}$, given in \cite{B}:
$$D_{2}: \hspace{.75in}\forall p\hspace{0.15in} \exists q\geq p+1 \hspace{0.15in} \forall k\geq q+1 \hspace{0.35in} \lim_{n\rightarrow \infty} {d_{n}\left( U_{q}, U_{p}\right) \over d_{n}\left( U_{k}, U_{q}\right) }=0.\hspace{1.35in}$$
\begin{prop}\label{p4} Let $E$ be a nuclear Fr\'echet space. The following are equivalent:
\begin{enumerate}
	\item $E$ has a prominent bounded set B.
	\item $E$ has the property $D_{2}$.
	\item For all p there exists $q>p$ such that $\displaystyle\sup_{l\geq q}\limsup_{n\in \mathbb{N}} {\varepsilon_{n}\left(q,l \right)\over \varepsilon_{n}\left(p,q \right) }\leq 1 $ holds.  
\end{enumerate}
\end{prop} 
We need the following lemma for the proof of Proposition \ref{p4}. As usual, we assume that all semi-norms are Hilbertian.
\begin{lemma}\label{l} Let $E$ be a nuclear Fr\'echet space. Then for all $p, q>p$, there is a $s>q$ such that
$$ \lim_{n\rightarrow +\infty} {d_{n}\left(U_{s}, U_{p} \right)\over  d_{n}\left(U_{q}, U_{p} \right)}=0.$$
\end{lemma}
\begin{proof} This is an immediate consequence of Proposition 1.2 in \cite{L}.
\end{proof}
It is worth noting that, by using Lemma 4.6, the condition $D_{2}$ can also be stated as follows:
$$D_{2}: \hspace{.75in}\forall p\hspace{0.15in} \exists q\geq p+1 \hspace{0.15in} \forall k\geq q+1 \hspace{0.35in} \sup_{n\in \mathbb{N}} {d_{n}\left( U_{q}, U_{p}\right) \over d_{n}\left( U_{k}, U_{q}\right) }<+\infty.\hspace{1.35in}$$
We are now ready to give the proof of Proposition \ref{p4}.
\begin{proof**} $1\Rightarrow 2:$ This follows immediately from Lemma 4.6 and the definition of $D_{2}$. 
~\\ $2\Rightarrow 1:$ Follows from Proposition 5 of \cite{TD}.
~\\ $2\Leftrightarrow 3:$ Suppose $E$ has the condition $D_{2}$. Then, for all $p$, there exists a $q>p$ such that for all $k>q$  
\begin{equation*}
\begin{split}
\sup_{n\in \mathbb{N}} {d_{n}\left( U_{q}, U_{p}\right) \over d_{n}\left( U_{k}, U_{q}\right) }<\infty \hspace{0.15in} & \Leftrightarrow \hspace{0.15in} \exists M>0\hspace{0.15in} \forall n\in \mathbb{N} \hspace{0.25in}{d_{n}\left( U_{q}, U_{p}\right) \over d_{n}\left( U_{k}, U_{q}\right) }\leq M\hspace{0.35in} \\
& \Leftrightarrow \hspace{0.15in} \exists M>0\hspace{0.15in} \forall n\in \mathbb{N} \hspace{0.25in}\varepsilon_{n}\left(p,q\right)\geq -\ln M + \varepsilon_{n}\left(q,k\right) \\
& \Leftrightarrow \hspace{0.15in} \limsup_{n\in \mathbb{N}}{\varepsilon_{n}\left(q,k\right)\over \varepsilon_{n}\left(p,q\right)}\leq 1.
\end{split}
\end{equation*}
Hence we obtain for all $p$, there exists a $q>p$ such that
 $$\displaystyle\sup_{l\geq q}\limsup_{n\in \mathbb{N}} {\varepsilon_{n}\left(q,l \right)\over \varepsilon_{n}\left(p,q \right) }\leq 1.$$
\end{proof**}
As an easy consequence of Proposition \ref{p4} and (3.1)(\cite[Cor.\ 1.10]{A}) above, we obtain the following result which gives a relation between having prominent bounded subset and its approximate diametral dimension of a nuclear Fr\'echet space with the properties \underline{DN} and $\Omega$:
\begin{cor}\label{c1} Let $E$ be a nuclear Fr\'echet space with the properties \underline{DN} and $\Omega$ and $\varepsilon$ the associated exponent sequence. $\delta\left(E\right)=\delta\left(\Lambda_{1}\left(\varepsilon\right)\right)$ implies that $E$ has a prominent bounded subset.
\end{cor}
The following theorem is the main result of this section which says that Question \ref{q1} holds true provided $E$ has a prominent bounded subset:
\begin{thm}\label{t5} Let $E$ be a nuclear Fr\'echet space with the properties \underline{DN} and $\Omega$ and $\varepsilon$ the associated exponent sequence.  $\delta\left(E\right)=\delta\left(\Lambda_{1}\left(\varepsilon\right)\right)$ if and only if $E$ has a prominent bounded set  and $\Delta\left(E\right)=\Lambda_{1}\left(\varepsilon\right)$.
\end{thm}
\begin{proof} Let $E$ be a nuclear Fr\'echet space with a prominent bounded subset $B$. Then $E$ satisfies condition $D_{2}$
$$ \forall p\hspace{0.15in} \exists q\geq p+1 \hspace{0.15in} \forall k\geq q+1\hspace{0.15in}\exists C>0 \hspace{0.45in} \lim_{n\rightarrow \infty} {d_{n}\left( U_{q}, U_{p}\right) \over d_{n}\left( U_{k}, U_{q}\right) }=0$$
and, in particular, if we take $N=p$, $M=n=q$ and $m=k$, we get 
$$\displaystyle \forall N\hspace{0.075in}\exists\hspace{0.075in}M, n\hspace{0.075in}\forall m, \hspace{0.075in}\exists\hspace{0.075in} S>0:\hspace{0.075in} d_{n}\left(U_{n}, U_{N}\right)\leq S d_{n}\left(U_{m}, U_{M}\right)\hspace{0.25in} \forall n\in \mathbb{N}.$$ 
which means that $E$ satisfy the condition $(wQ)$ given in Theorem \ref{t1} and so $\Delta\left(E\right)$ is barrelled with respect to the canonical topology. Hence the result follows from Theorem \ref{t4}.
\end{proof}
In the final part of this section we examine the conditions for which the converse of Corollary \ref{c1} also holds.
~\\ For this, we define
$$\tdot\left(E\right):=\left\{\left(t_{n}\right)_{n\in \mathbb{N}}: \hspace{0.1in} \forall p, \hspace{0.1in}\forall 0< \varepsilon<1, \hspace{0.1in} \exists q>p \hspace{0.1in}\lim_{n\rightarrow +\infty} t_{n}d_{n}\left(U_{q},U_{p}\right)^{\varepsilon}=0\right\}.$$
The next result provides a condition that implies $\delta\left(E\right)=\delta\left(\Lambda_{1}\left(\varepsilon\right)\right)$ when $E$ has a prominent subset.  
\begin{prop}\label{p5} Let $E$ be a nuclear Fr\'echet space with the properties $\underline{DN}$ and $\Omega$, its associated exponent sequence $\varepsilon=\left(\varepsilon_{n}\right)_{n\in \mathbb{N}}$. If $E$ has a prominent bounded subset $B$ and $\Delta\left(E\right)= \tdot\left(E\right)$, then $\delta\left(E\right)= \delta\left(\Lambda_{1}\left(\varepsilon\right)\right)$. 
\end{prop}
\begin{proof} Let $B$ be a prominent bounded subset of $E$. Then, for all $p\in \mathbb{N}$, there exists a $q>p$ and a $C>0$ so that for every $n\in \mathbb{N}$
$$d_{n}\left(U_{q}, U_{p}\right)\leq C d_{n}\left(B, U_{q}\right)$$
holds. Also, since $B$ is bounded and $\varepsilon_{n}$ is the associated exponent sequence, then there exist $C_{1}, C_{2}, D_{1}, D_{2}>0$ and $a_{1}, a_{2}>0$ satisfying
$$D_{1} e^{{-a_{1}\varepsilon_{n}}}\leq C_{1}d_{n}\left(U_{q}, U_{p}\right)\leq d_{n}\left(B, U_{q}\right)\leq C_{2}d_{n}\left(U_{q+1}, U_{q}\right)\leq D_{2} e^{{-a_{2}\varepsilon_{n}}}$$
for every $n\in \mathbb{N}$. On the other hand, $\displaystyle \Delta\left(E\right)=\left\{\left(x_{n}\right)_{n\in \mathbb{N}}: \hspace{0.1in}\forall p \hspace{0.1in} \lim_{n\rightarrow +\infty} \left|x_{n}\right|d_{n}\left(B, U_{p}\right)=0\right\}$ is a Fr\'echet space since B is a prominent set. Fix $p$, $q>p$ and $\varepsilon$. Consider the Banach space
$$\displaystyle B_{p,\varepsilon, q}= \left\{t=\left(t_{n}\right)_{n\in \mathbb{N}}: \hspace{0.1in}\lim_{n\rightarrow +\infty} \left|t_{n}\right|d_{n}\left(U_{q},U_{p}\right)^{\varepsilon}=0\right\}.$$
Since $U_{q+1}\subseteq U_{q}$, we have $d_{n}\left(U_{q+1}, U_{p}\right)\leq d_{n}\left(U_{q}, U_{p}\right)$ for every $n\in \mathbb{N}$ and $B_{p, \varepsilon, q}\subseteq  B_{p,\varepsilon, q+1}$. Then we can put the topology on $\displaystyle \tdot\left(E\right)= \bigcap_{\left(p,\varepsilon\right)} \bigcup_{q>p} B_{p, \varepsilon,q}$ which of the projective limit of inductive limits of Banach spaces $B_{p,\varepsilon, q}$. In view of Grothendieck Factorization theorem (\cite{K}, p.225), for all $p$, $0< \varepsilon<1$ there exists a $q>p$ such that $\Delta\left(E\right) \hookrightarrow B_{p, \varepsilon,q}$ is continous 
$$\forall p, \hspace{0.15in} 0<\varepsilon< 1, \hspace{0.15in} \exists q>p, \hspace{0.25in} \hspace{0.05in}C>0\hspace{0.25in} d_{n}\left(U_{q}, U_{p}\right)^{\varepsilon}\leq d_{n}\left(B, U_{q}\right)\hspace{0.25in} \forall n\in \mathbb{N}.$$
Now take $\delta>0$. Then, for a given $p$, we choose $0<\varepsilon <1$ so that $\displaystyle 0<\varepsilon <{\delta \over a_{1}}$. Then there exists a $\overline{C}>0$ so that for all $n\in \mathbb{N}$,
\begin{equation*}
\begin{split}
\overline{C} e^{{-\varepsilon a_{1}\varepsilon_{n}} }\leq C d_{n}\left(B, U_{p}\right)^{\varepsilon}\leq d_{n}\left(U_{q}, U_{p}\right)^{\varepsilon} \hspace{0.2in} & \Leftrightarrow \hspace{0.2in} \overline{C}e^{-\delta\varepsilon_{n}}\leq  d_{n}\left(U_{q}, U_{p}\right)^{\varepsilon}  \leq C d_{n}\left(B, U_{q}\right) \\
&\Leftrightarrow \hspace{0.2in} \ln{\overline{C}} -\delta \varepsilon_{n} \leq \ln{C} +\ln{d_{n}\left(B,U_{q}\right)} \leq \ln{C}+ \ln{d_{n}\left(U_{l},U_{q}\right)} \\
& \Rightarrow \hspace{0.2in} -\ln{d_{n}\left(U_{l}, U_{q}\right)}\leq \left(\ln{C}- \ln{\overline{C}}\right)+\delta \varepsilon_{n}\\
& \Rightarrow \hspace{0.2in} \limsup_{n} {\varepsilon_{n}\left(q,l\right)\over \varepsilon_{n}}\leq \delta.
\end{split}
\end{equation*}
Hence, we obtain that for all $\delta>0$ there is a $q$ so that $$\displaystyle \sup_{l>q}\limsup_{n} {\varepsilon_{n}\left(q,l\right)\over \varepsilon_{n}}\leq \delta \hspace{0.35in}\textnormal{and} \hspace{0.35in} \displaystyle \inf_{q} \sup_{l>q}\limsup_{n} {\varepsilon_{n}\left(q,l\right)\over \varepsilon_{n}}=0,$$ which means that $\delta\left(E\right)=\delta\left(\Lambda_{1}\left(\varepsilon\right)\right)$.
\end{proof}
Note that $\tdot\left(E\right)$ is always an algebra under multiplication. If $\left(t_{n}\right)_{n\in \mathbb{N}}\in \tdot\left(E\right)$, then for all $p$, $0<\varepsilon<1$, we can choose $q>p$ so that
$$\displaystyle \lim_{n\rightarrow\infty}t_{n}d_{n}\left(U_{q},U_{p}\right)^{{{\varepsilon\over 2}}}=0,$$
which means $\left(t^{2}_{n}\right)\in \tdot\left(E\right)$. Then for any $\left(t_{n}\right)_{n\in \mathbb{N}}, \left(s_{n}\right)_{n\in \mathbb{N}}\in \tdot\left(E\right)$, we have that $\left(t_{n}s_{n}\right)_{n\in \mathbb{N}}\in \tdot\left(E\right)$ as $\displaystyle \left|t_{n}s_{n}\right|\leq {\left|t_{n}\right|^{2}\over 2}+{\left|s_{n}\right|^{2}\over 2}$ for all $n\in \mathbb{N}$.
~\\ But $\Delta (E)$ need not to be an algebra under multiplication. If it does, then $\Delta\left(E\right)$ satisfies the condition "$\left(t_{n}\right)\in \Delta\left(E\right)$ implies $\left(t^{2}_{n}\right)\in \Delta\left(E\right)$", vice versa. This condition gives that $\left(t^{2^{m}}_{n}\right)\in \Delta\left(E\right)$ for all $m\in \mathbb{N}$. Now, for a $p$ and $\varepsilon>0$, we can choose $m\in \mathbb{N}$ so large that $\displaystyle {1\over 2^{m}}\leq \varepsilon$ and find a q so that
$$\displaystyle \lim_{n\rightarrow \infty} t^{2^{m}}_{n}d_{n}\left(U_{q}, U_{p}\right)=0 \hspace{0.25in} \textnormal{and} \hspace{0.25in}\displaystyle \lim_{n\rightarrow \infty} t_{n}d_{n}\left(U_{q}, U_{p}\right)^{\varepsilon}=0$$
which gives that $\Delta\left(E\right)\subseteq \tdot\left(E\right)$. Since $\tdot\left(E\right) \subseteq \Delta\left(E\right)$ always holds, we have $\tdot\left(E\right) =\Delta\left(E\right)$. Hence we conclude that the followings are equivalent:
\begin{enumerate}
\item $\Delta\left(E\right)=\tdot\left(E\right)$
\item $\Delta\left(E\right)$ is an algebra under multiplication.
\item $\left(t_{n}\right)\in \Delta\left(E\right)$ implies $\left(t^{2}_{n}\right)\in \Delta\left(E\right)$.
\end{enumerate}
We end this paper with the following result which is a generalization of Proposition \ref{p5}:
\begin{cor} Let $E$ be a nuclaer Fr\'echet space with the properties \underline{DN} and $\Omega$, its associated exponent sequence  $\varepsilon= \left(\varepsilon_{n}\right)_{n\in \mathbb{N}}$. E has a prominent bounded subset and $\Delta\left(E\right)$ is an algebra if and only if $\delta\left(E\right)= \delta\left(\Lambda_{1}\left(\varepsilon\right)\right)$.
\end{cor}

%%%%%%%%%%%%%%%%%%%%%
\textbf{Acknowledgements. }
The results in this paper are from the author's thesis, supervised by  Ayd{\i}n Aytuna. We would like to thank Prof. Aytuna for suggesting this topic and his constant guidance and insight.
\footnotesize{{}
\end{document}